\documentclass[11pt]{amsart}
\usepackage{geometry}
\usepackage{empheq}        
\usepackage{amssymb}       
\usepackage{amsthm}        
\usepackage{esint}         
\usepackage{dsfont}        
\usepackage{mathrsfs}      
\usepackage{cases}

\usepackage{xcolor}        
\usepackage{graphicx}      
\usepackage{tikz}          
\usetikzlibrary{cd}        

\usepackage{array}
\usepackage{enumitem}
\usepackage{tabularx, booktabs}

\usepackage[abbrev]{amsrefs}

\usepackage[colorlinks=true, allcolors=blue]{hyperref}
\hypersetup{
    colorlinks = true,
	allcolors = blue
}

\theoremstyle{plain}
\newtheorem{thm}{Theorem}[section]
\newtheorem{cor}[thm]{Corollary}

\newtheorem{lem}[thm]{Lemma}

\theoremstyle{definition}

\newtheorem{rem}[thm]{Remark}

\numberwithin{equation}{section}

\DeclareMathOperator\vol{vol}
\DeclareMathOperator\dive{div}

\DeclareMathOperator\sn{sn_\delta}
\DeclareMathOperator\cn{cn_\delta}
\DeclareMathOperator\tn{tn_\delta}

\begin{document}
\title{Extrinsic upper bounds for the harmonic mean of  eigenvalues in curved manifolds}

\author{Hang Chen}
\address[Hang Chen]{School of Mathematics and Statistics, Northwestern Polytechnical University, Xi' an 710129, P. R. China \\ email: chenhang86@nwpu.edu.cn}
\thanks{Chen is supported by NSFC Grant No.~12571054}
\author{Yuna Gao}
\address[Yuna Gao]{School of Mathematics and Statistics, Northwestern Polytechnical University, Xi' an 710129, P. R. China \\ email: gaoyuna@mail.nwpu.edu.cn}

\begin{abstract}
	Let $M$ be an $m (\ge2)$-dimensional closed orientable submanifold in a Riemannian manifold of sectional curvature bounded above by $\delta$.
	We obtain the extrinsic upper bounds of the harmonic mean of the first $m$ nonzero eigenvalues of  the Laplacian on $M$, which extend the previous related results for the first nonzero eigenvalues.
\end{abstract}


\maketitle

\section{Introduction}

Reilly-type inequalities are important estimates in differential geometry.
They indicate that the first nonzero eigenvalue of  the Laplacian on a submanifold can be controlled by the extrinsic curvature of the submanifold in the ambient space.
The simplest version can be stated as follows.
\begin{thm}[\cites{Rei77, ESI92}]\label{thm-Reilly}
	Let $M$ be an $m(\ge 2)$-dimensional closed orientable submanifolds in a space form $\mathbb{R}^{n}(c)$. Then the first nonzero eigenvalue $\lambda_1$  of the Laplacian on $M$ satisfies
	\begin{equation}\label{eq-Rei}
		\lambda_1\leq\frac{m}{\vol(M)}\int_M(c+|\mathbf{H}|^2),
	\end{equation}
	where $\mathbf{H}$ is the mean curvature vector of $M$ in $\mathbb{R}^{n}(c)$.

	Equality holds if and only if $M$ is either minimally immersed in $\mathbb{S}^n$ by a subspace of the first eigenspace, or minimally immersed in a geodesic sphere of $\mathbb{R}^{n}(c)$.
\end{thm}

This inequality was first proved by R. Reilly in \cite{Rei77} for the case $c=0$.
The case $c=1$ can be reduced to the case $c=0$ by embedding the sphere $\mathbb{S}^n$ into the Euclidean space $\mathbb{R}^{n+1}$.
However, the method fails for the case $c=-1$.
El Soufi and Ilias \cite{ESI92} proved \eqref{eq-Rei} in a unified way for all $c$ by constructing test functions via conformal transformations from $\mathbb{R}^n(c)$ to the unit sphere $\mathbb{S}^n(1)$.
Under the same setting as in Theorem \ref{thm-Reilly}, \eqref{eq-Rei} has been generalized to some other operators, such as the Schr\"odinger operator \cite{ESI00}, the $L_T$ operator \cites{Gro00, CW19a}, the $p$-Laplacian \cite{CW19a}, the Paneitz operator \cite{CL11}, etc.

On the other hand, one can replace the ambient space $\mathbb{R}^n(c)$ by a curved Riemannian manifold $N$.
Precisely, we have the following the theorem.
\begin{thm}[\cites{Hei88, NX21}]\label{thm-Hei-NX}
	Let $M$ be an $m(\ge 2)$-dimensional closed orientable submanifold in an $n$-dimensional Riemannian manifold $N$ of the sectional curvature $K_N \le \delta$.
	If $\delta\le 0$, we assume that $N$ is simply-connected;
	if $\delta> 0$, we assume that $M$ is contained in a convex ball of radius $\le \pi/4\sqrt{\delta}$.
	Then the first nonzero eigenvalue $\lambda_1$ of the Laplacian on $M$ satisfies
	\begin{equation}\label{eq-Hei-NX}
		\lambda_1\le \frac{m}{\vol(M)}\int_M\big(\delta+|\mathbf{H}|^2\big).
	\end{equation}

	Equality holds implies that $M$ is minimally immersed in a geodesic sphere of $N$.
\end{thm}

Heintze \cite{Hei88} proved \eqref{eq-Hei-NX} for $\delta\ge 0$.
However, for $\delta<0$, the approach in \cite{Hei88} only gives a weaker estimate in term of $L^{\infty}$-norm of the mean curvature rather than $L^2$-norm of the mean curvature, i.e.,
\begin{equation}
	\lambda_1\le m\big(\delta+\max |\mathbf{H}|^2\big).
\end{equation}
Recently, Niu and Xu \cite{NX21} proved \eqref{eq-Hei-NX} for $\delta<0$ by constructing new test functions.

Under the same setting as in Theorem \ref{thm-Hei-NX}, the estimates \eqref{eq-Hei-NX} also have been generalized to the $L_T$ operator and the $p$-Laplacian (cf. \cites{AdCM01, Gro04, Che20, CG25}).

Very recently, Chen \cite{Che26} obtained the sharp extrinsic upper bounds of the harmonic mean of the first $m$ nonzero eigenvalues of the Laplacian on $M$ in $\mathbb{R}^n(c)$, namely,
\begin{thm}[\cite{Che26}*{Theorem 1.13}]
	Let $M$ be an $m$-dimensional closed submanifold in $\mathbb{R}^n(c)$. Then the first $m$ nonzero eigenvalues of the Laplacian satisfy
\begin{equation}\label{eq_Rei_harmonic}
	\mathfrak{H}(\lambda_1,\cdots, \lambda_m) \le \frac{m}{\vol(M)}\int_M\left(c+|\mathbf{H}|^2\right),
\end{equation}
where
\begin{equation*}
	\mathfrak{H}(\lambda_1,\cdots,\lambda_m)=\Big(\frac{1}{m}\sum_{i=1}^m\lambda_i^{-1}\Big)^{-1}
\end{equation*}
represents the harmonic mean of the first $m$ nonzero eigenvalues $\lambda_1,\cdots,\lambda_m$ of the Laplacian on $M$.

Equality holds if and only if $M$ is either minimally immersed in $\mathbb{S}^n$ by a subspace of the first eigenspace, or minimally immersed in a geodesic sphere of $\mathbb{R}^{n}(c)$.

\end{thm}

Since $\mathfrak{H}(\lambda_1,\cdots,\lambda_m)\ge \lambda_1$, \eqref{eq_Rei_harmonic} indeed improves \eqref{eq-Rei}; it means that if $\lambda_1$ is close to the upper bound, then $\lambda_m$ cannot be far away from the same upper bound.

Inspired by the work \cite{Che26}, we will extend the estimates \eqref{eq-Hei-NX} to the harmonic mean of the first $m$ nonzero eigenvalues of the Laplacian on $M$ in a curved Riemannian manifold $N$.
We prove the following theorems.
\begin{thm}\label{thm-nonpositive}
	Let $M$ be an $m(\ge 2)$-dimensional closed orientable submanifold in an $n$-dimensional simply-connected Riemannian manifold $N$ of the sectional curvature $K_N \le \delta\le 0$.

	Then we have
	\begin{equation*}
		\mathfrak{H}(\lambda_1,\cdots,\lambda_m)\le \frac{m}{\vol(M)}\int_M\big(\delta+|\mathbf{H}|^2\big).
	\end{equation*}
	Equality holds implies that $\lambda_1=\cdots=\lambda_m$ and $M$ is minimally immersed in a geodesic sphere of $N$.
\end{thm}

\begin{thm}\label{thm-positive}
	Let $M$ be an $m(\ge 2)$-dimensional closed orientable submanifold in an $n$-dimensional Riemannian manifold $N$ of the sectional curvature $K_N \le \delta$, where $\delta > 0$.
	We assume that $M$ is contained in a convex ball of radius $\le \frac{1}{2\sqrt{\delta}}\arctan \frac{m}{\sqrt{2m+1}}$.
	Then we have
	\begin{equation*}
		\mathfrak{H}(\lambda_1,\cdots,\lambda_m)\le \frac{m}{\vol(M)}\int_M\big(\delta+|\mathbf{H}|^2\big).
	\end{equation*}
	Equality holds implies that $\lambda_1=\cdots=\lambda_m$ and $M$ is minimally immersed in a geodesic sphere of $N$.

\end{thm}

\begin{cor}
	Settings as in Theorem \ref{thm-nonpositive} or Theorem \ref{thm-positive}.
	If
	\begin{equation*}
		\lambda_1= \frac{m}{\vol(M)}\int_M\big(\delta+|\mathbf{H}|^2\big),
	\end{equation*}
 then the multiplicity of $\lambda_1$ is at least $m$.
\end{cor}

\section{Preliminaries}
In this section, we introduce some notations and review some lemmas.

From now on, we always assume that $M$ is an $m$-dimensional submanifold of an $n$-dimensional Riemannian manifold $N$ whose sectional curvature $K_N\le \delta$.
We denote the Levi-Civita connection on $M$ and $N$ by $\nabla$ and $\nabla^N$ respectively.

The following functions naturally arise in the study of Jacobi fields along geodesics in a space form of constant curvature $\delta$.
\begin{gather}
\sn(t)=
\begin{cases}
\frac{1}{\sqrt{\delta}}\sin(\sqrt{\delta}t), & \mbox{for } \delta>0;\\
t, & \mbox{for }  \delta=0;\\
\frac{1}{\sqrt{-\delta}}\sinh(\sqrt{-\delta}t), & \mbox{for }  \delta<0.
\end{cases}\nonumber
\end{gather}
Actually, $\sn(t)$ is the unique solution of the differential equation $y''(t)+\delta y(t)=0$ with the initial conditions $\sn(0)=0 $ and $\sn'(0)=1$.
Now set $\cn(t)=\sn'(t)$ and $\tn(t)=\sn(t)/\cn(t)$, then one can easily check that $\cn^2+\delta\sn^2=1$ and $1/\cn^2=1+\delta\tn^2$.

Given a fixed point $q_0\in N$, for any $q\in N$,
let $r=r(q)$ be the distance function from $q_0$ to $q$ and $(x_1,\cdots,x_n)$ be the normal coordinates of $q$ centered at $q_0$.

By using basic properties of the exponential map and the comparison for Jacobi fields, we have the following lemma.
\begin{lem}[cf. \cite{Gro04}*{Lemmas 1 and 2} and \cite{Che26}*{Lemma 2.1}]\label{lem-2.1}
	Let $\phi(r)$ be a positive function with $\phi(r) \sim r$ as $r\to 0$.
	Then we have
	\begin{gather}
		\sum_{i=1}^n \Big|\nabla\big(\frac{\phi}{r}x_i\big)\Big|^2\le m\frac{\phi^2}{\sn^2}+ \Big((\phi')^2-\frac{\phi^2}{\sn^2}\Big)|\nabla r|^2;\label{eq-2.1}\\
		\dive (\phi\nabla r)\ge m \frac{\phi}{\tn}+\Big(\phi'-\frac{\phi}{\tn}\Big)
		|\nabla r|^2+m\langle \phi\nabla^Nr, \mathbf{H}\rangle.		\label{eq-2.2}
	\end{gather}

	If $N$ has constant sectional curvature $\delta$, then the equality holds in \eqref{eq-2.1} and \eqref{eq-2.2}.
\end{lem}

\begin{rem}
	If we denote the vector field $\phi(r)\nabla^N r$ by $Z$, then $\big(\dfrac{\phi(r)}{r}x_i\big)_{1\leq i\leq n}$ are the coordinates of $Z$ in the normal local frame, and $\phi\nabla r=Z^\top$ is the tangential projection of $Z$ to $M$.

	Hence, the left hand sides of \eqref{eq-2.2} is just $\dive_M(Z^\top)$.
\end{rem}

For estimating the harmonic mean of the eigenvalues, we also need to control each individual term in the sum on the left hand side of \eqref{eq-2.1}.
\begin{lem}\label{lem-2.3}
	Let $\phi(r)$ be a positive function with $\phi(r) \sim r$ as $r\to 0$.
	Then we have
	\begin{equation*}
		\Big|\nabla\big(\frac{\phi(r)}{r}x_i\big)\Big|^2\le \frac{\phi^2}{\sn^2}(r)+\Big(\phi'^2(r)-\frac{\phi^2}{\sn^2}(r)\Big)\big(\frac{x_i}{r}\big)^2.
	\end{equation*}
\end{lem}
\begin{proof}
	A direct computation gives
	\begin{align}
		\Big|\nabla\big(\frac{\phi(r)}{r}x_i\big)\Big|^2
		&\le \Big|\nabla^N \big(\frac{\phi(r)}{r}x_i\big)\Big|^2=\Big|\frac{x_i}{r}\nabla^N\phi(r)+\phi(r)\nabla^N \big(\frac{x_i}{r}\big)\Big|^2\nonumber\\
		&=\phi'^2(r) \big(\frac{x_i}{r}\big)^2+\phi^2(r)\Big|\nabla^{\Sigma_r} \big(\frac{x_i}{r}\big)\Big|^2\label{eq-lem2.3-1}\\
		&\le \phi'^2(r) \big(\frac{x_i}{r}\big)^2+\frac{\phi^2}{\sn^2}(r)\Big|\nabla^{S^{n-1}} \big(\frac{x_i}{r}\big)\Big|^2\label{eq-lem2.3-2}\\
		&=\phi'^2(r) \big(\frac{x_i}{r}\big)^2+\frac{\phi^2}{\sn^2}(r)\Big(1-\big(\frac{x_i}{r}\big)^2\Big)\nonumber\\
		&=\frac{\phi^2}{\sn^2}(r)+\Big(\phi'^2(r)-\frac{\phi^2}{\sn^2}(r)\Big)\big(\frac{x_i}{r}\big)^2.\nonumber
	\end{align}
	We used $|\nabla^N r|=1$ and the Gauss lemma in \eqref{eq-lem2.3-1}, since $x_i/r$ does not depend on $r$. \eqref{eq-lem2.3-2} follows from the Rauch comparison theorem, where $\Sigma_r$ is the geodesic sphere in $N$ of radius $r$ centered at $q_0$, and $S^{n-1}$ is the unit sphere in the Euclidean space $\mathbb{R}^n$.

\end{proof}

\section{Proof of Main Theorems}
It is well known that the eigenvalues of Laplacian on $M$ are discrete and satisfies
\begin{equation*}
	0=\lambda_0<\lambda_1\le \lambda_2\le \cdots \to +\infty.
\end{equation*}
Let $\{u_i\}_{i\ge 0}$ is an orthonormal set of eigenfunctions.
According to the variational characterization of $\lambda_i (i\ge 1)$:
\begin{equation}\label{eq-vc}
	\lambda_i=\inf_{u\in H^1(M)\setminus \{0\}}\left\{\frac{\int_M |\nabla u|^2}{\int_M u^2}\right|\left.\int_M uu_j=0, j=0, \cdots, i-1\right\},
\end{equation}
let us construct the test functions.

The following lemma provides suitable test functions for the upper bound estimates of the first eigenvalue.
Special versions can be traced back to \cites{Cha78, Hei88}.
	\begin{lem}[cf. \cite{Che26}*{Lemma 2.3}]\label{lem-3.1}
		Let $\phi(r)$ be a positive function with $\phi(r) \sim r$ as $r\to 0$.
		Under the assumptions as in Theorem \ref{thm-nonpositive} or Theorem \ref{thm-positive}, there exists a point $q_0 \in N$ such that
		\begin{equation}\label{eq-3.1}
			\int_M \frac{\phi(r)}{r} x_i\, dv_M = 0 \, (i=1,\cdots,n).
		\end{equation}
	\end{lem}

	\begin{rem}\label{rem-3.1}
		If $M$ is contained in a convex ball of radius $R$, then the distant function $r(\cdot)=d(\cdot,q_0)\le 2R$ on $M$.
	\end{rem}

	For convienience, we denote $\Phi_i=\frac{\phi(r)}{r}x_i$ for $i=1,\cdots,n$.
	We claim that each $\Phi_i$ can be assumed to be a test function for $\lambda_i$.

	Firstly, \eqref{eq-3.1} means that $\Phi_i$ is $L^{2}$-orthogonal to the first eigenfunction $u_0$ (a nonzero constant).
	Next, if we denote
	\begin{equation*}
		d_{ij}=\int_M \Phi_i u_j,  \mbox{ for } 1\le i, j\le n,
	\end{equation*}
	then the matrix $D=(d_{ij})=QR$ via the QR-decomposition, where $Q=(q_{ij})\in O(n)$ and $R$ is an upper triangular matrix.
	Equivalently, we have $R=Q^TD$.
	Hence, under a new normal coordinate system (still denoted by $\{x_i\}_{1\le i \le n}$) obtained by an othorogonal transformation,
	we have
	\begin{equation}\label{eq-ortho-i}
		\int_M \Phi_i u_j=0,  \mbox{ for } 1\le j< i\le n.
	\end{equation}

	Now, by the variational characterization \eqref{eq-vc}, we have
	\begin{equation*}
		\lambda_i\int_M \Phi_i^2\le \int_M | \nabla \Phi_i|^2.
	\end{equation*}
	This gives
	\begin{equation}\label{eq-3.2}
		\begin{aligned}
			\int_M \phi(r)^2\le {} &  \sum_{i=1}^{n} \frac{1}{\lambda_i} \int_M |\nabla\Phi_i|^2= \sum_{i=1}^{m} \frac{1}{\lambda_i} \int_M |\nabla\Phi_i|^2+\sum_{i=m+1}^{n} \frac{1}{\lambda_i} \int_M |\nabla\Phi_i|^2\\
			\le{} & \sum_{i=1}^{m} \frac{1}{\lambda_i} \int_M |\nabla\Phi_i|^2+ \frac{1}{\lambda_m}	\int_M \sum_{i=m+1}^{n}|\nabla\Phi_i|^2 \\
			\le{} & \sum_{i=1}^{m} \frac{1}{\lambda_i} \int_M |\nabla\Phi_i|^2 + \frac{1}{\lambda_m} \int_M \Big(m\frac{\phi^2}{\sn^2}+ \big((\phi')^2-\frac{\phi^2}{\sn^2}\big)|\nabla r|^2-\sum_{i=1}^{m}|\nabla\Phi_i|^2 \Big) \\
			={} & \sum_{i=1}^{m} \frac{1}{\lambda_i} \int_M |\nabla\Phi_i|^2 + \frac{1}{\lambda_m} \sum_{i=1}^{m} \int_M \Big(\frac{\phi^2}{\sn^2}-|\nabla\Phi_i|^2 \Big)+\frac{1}{\lambda_m}\int_M \big((\phi')^2-\frac{\phi^2}{\sn^2}\big)|\nabla r|^2.
		\end{aligned}
	\end{equation}

	If $(\phi')^2-\frac{\phi^2}{\sn^2}\le 0$,
	then $|\nabla\Phi_i|^2\le \frac{\phi^2}{\sn^2}$ by Lemma \ref{lem-2.3}.
	Hence, \eqref{eq-3.2} gives
	\begin{equation*}
		\int_M \phi(r)^2 \le \sum_{i=1}^{m} \frac{1}{\lambda_i} \int_M |\nabla\Phi_i|^2 +  \sum_{i=1}^{m} \frac{1}{\lambda_i}\int_M \Big(\frac{\phi^2}{\sn^2}-|\nabla\Phi_i|^2 \Big)= \sum_{i=1}^{m}\frac{1}{\lambda_i} \int_M \frac{\phi^2}{\sn^2},
	\end{equation*}
	equivalently,
	\begin{equation}\label{eq-init-estimate}
		\mathfrak{H}(\lambda_1,\cdots,\lambda_m)\int_M \phi(r)^2\le \int_M m\frac{\phi^2}{\sn^2}.
	\end{equation}


\subsection{The case $\delta=0$.}
We have $\sn(r)=\tn(r)=r$.
We take $\phi(r)=r$, then $(\phi')^2-\frac{\phi^2}{\sn^2}=0$ and
\eqref{eq-init-estimate} becomes
\begin{equation}\label{eq-zero-1}
	\mathfrak{H}(\lambda_1,\cdots,\lambda_m)\int_M r^2\le \int_M m = m\vol(M).
\end{equation}

By \eqref{eq-2.2}, we have
\begin{equation*}
	0=\int_M \dive_M(Z^\top) \ge m \vol(M)+m\int_M \langle r\nabla^N r, \mathbf{H}\rangle.
\end{equation*}
Hence,
\begin{equation}\label{eq-zero-2}
	\vol(M)^2 \le \left(\int_M \langle r\nabla^N r, \mathbf{H}\rangle\right)^2 \le \left(\int_M |r\nabla^N r| |\mathbf{H}|\right)^2 \le \int_M r^2 \int_M |\mathbf{H}|^2.
\end{equation}

Combining \eqref{eq-zero-1} and \eqref{eq-zero-2}, we obtain
\begin{equation*}
	\mathfrak{H}(\lambda_1,\cdots,\lambda_m)\le \frac{m}{\vol(M)}\int_M |\mathbf{H}|^2.
\end{equation*}

\subsection{The case $\delta<0$.}
We take $\phi(r)=\tn(r)$, then $(\phi')^2-\frac{\phi^2}{\sn^2}\le 0$ and
\eqref{eq-init-estimate} gives
\begin{equation}\label{eq-neg-1}
	\mathfrak{H}(\lambda_1,\cdots,\lambda_m)\int_M \tn(r)^2\le m \int_M \frac{1}{\cn(r)^2} = m \int_M (1+\delta\tn(r)^2),
\end{equation}

Recalling the following key inequality (\cite{NX21}*{Proposition 2.1})
\begin{equation*}
	\vol(M)^2\le \int_M |\mathbf{H}|^2\int_M \tn(r)^2,
\end{equation*}
we obtain from \eqref{eq-neg-1} that
\begin{equation*}
	\mathfrak{H}(\lambda_1,\cdots,\lambda_m)\le \frac{m}{\vol(M)}\int_M (|\mathbf{H}|^2+\delta).
\end{equation*}

\subsection{The case $\delta>0$.}
We take $\phi(r)=\sn(r)$, then $(\phi')^2-\frac{\phi^2}{\sn^2}=-\delta\sn^2\le 0$ and
\eqref{eq-init-estimate} gives
\begin{equation}\label{eq-pos-1}
	\mathfrak{H}(\lambda_1,\cdots,\lambda_m)\int_M \sn^2(r)\le m \int_M 1 = m\vol(M).
\end{equation}

If we can prove
\begin{equation}\label{eq-pos-key}
	\vol(M)^2  \le \int_M (|\mathbf{H}|^2+\delta) \int_M \sn^2(r),
\end{equation}
then it follows from \eqref{eq-pos-1} that
\begin{equation*}
	\mathfrak{H}(\lambda_1,\cdots,\lambda_m)\le \frac{m}{\vol(M)}\int_M (|\mathbf{H}|^2+\delta).
\end{equation*}
Hence, we will prove \eqref{eq-pos-key} in the following.

Since $0\le r \le \pi/2\sqrt{\delta}$,
from \eqref{eq-2.2} we derive that
\begin{align*}
    \int_M \cn^2(r)&\le \frac{1}{m}\int_M \cn(r)\dive (\sn(r)\nabla r)- \int_M \langle \sn(r)\cn(r)\nabla^N r, \mathbf{H}\rangle \\
    & = -\frac{1}{m} \int_M \langle \nabla \cn(r), \sn(r)\nabla r\rangle - \int_M \langle \sn(r)\cn(r)\nabla^N r, \mathbf{H}\rangle  \\
    & = \frac{1}{m}\int_M \delta\sn^2(r)|\nabla r|^2 - \int_M \langle \sn(r)\cn(r)(\nabla^N r)^{\perp}, \mathbf{H}\rangle.
\end{align*}

    By using $\delta \sn^2+ \cn^2=1$,
    we have
    \begin{align*}
        \vol(M) & = \int_M \Big(\delta\sn^2(r)+\frac{1}{m}\delta\sn^2(r)|\nabla r|^2 -\sn(r)\cn(r)\langle (\nabla^N r)^{\perp}, \mathbf{H}\rangle\Big) \\
        & \le \int_M \sn(r)\Big(\delta\sn(r)\big(1+\frac{|\nabla r|^2}{m}\big) +\cn(r)|(\nabla^N r)^{\perp}|| \mathbf{H}|\Big) \\
        & \le \int_M \sn(r)\Big(\delta\sn^2(r)\big(1+\frac{|\nabla r|^2}{m}\big)^2 +\cn^2(r)|(\nabla^N r)^{\perp}|^2\Big)^{1/2} \Big(\delta+|\mathbf{H}|^2\Big)^{1/2}\\
        & \le \Big(\int_M \sn^2(r)\Big(\delta\sn^2(r)\big(1+\frac{|\nabla r|^2}{m}\big)^2 +\cn^2(r)|(\nabla^N r)^{\perp}|^2\Big)\Big)^{1/2} \Big(\int_M (\delta+|\mathbf{H}|^2)\Big)^{1/2}.
    \end{align*}
    where we used the Cauchy-Schwarz inequality and the H\"{o}lder inequality in the second inequality and the last inequality, respectively.

  Since $1=|\nabla^N r|^2 = |\nabla r|^2 + |(\nabla^N r)^{\perp}|^2$, we have
    \begin{align*}
        & \delta\sn^2(r)\big(1+\frac{|\nabla r|^2}{m}\big)^2 +\cn^2(r)|(\nabla^N r)^{\perp}|^2
        \\
        ={}& \delta\sn^2(r)\big(1+\frac{|\nabla r|^2}{m}\big)^2 +\cn^2(r)(1-|\nabla r|^2) \\
        ={}& 1 + |\nabla r|^2 \big(\frac{2\delta\sn^2(r)}{m}-\cn^2(r)+\frac{\delta\sn^2(r)}{m^2}|\nabla r|^2\big)  \\
        \le {}& 1 + |\nabla r|^2 \Big(\delta\sn^2(r)\big(\frac{2}{m}+\frac{|\nabla r|^2}{m^2}\big)-\cn^2(r)\Big)  \\
        \le {}& 1 + |\nabla r|^2 \Big(\delta\sn^2(r)\frac{2m+1}{m^2}-\cn^2(r)\Big)  \\\
        = {}& 1 + |\nabla r|^2 \cn^2(r)\Big(\tan^2(\sqrt{\delta}r)\frac{2m+1}{m^2}-1\Big)  \\
        \le{} & 1
    \end{align*}
    provided that $r\le \frac{1}{\sqrt{\delta}}\arctan \frac{m}{\sqrt{2m+1}}$.
    According to Remark \ref{rem-3.1}, it indeed holds by the assumption.
    Thus, we obtain
    \begin{equation*}
        \vol(M) \le \Big(\int_M \sn^2(r)\Big)^{1/2} \Big(\int_M (\delta+|\mathbf{H}|^2)\Big)^{1/2},
    \end{equation*}
    which is equivalent to \eqref{eq-pos-key}.
    We complete the proof of Theorem \ref{thm-positive}.

\subsection{Equality case.}
We discuss the equality case in a unified way.

If the equality holds, we can check the proofs and see that $\nabla^N r$ must be parallel to $\mathbf{H}$.
Thus, $M$ lies in a geodesic sphere of $N$ since $\nabla r=0$ and the mean curvature of $M$ in this geodesic sphere is zero.

On the other hand, checking process of the proofs, we have
\begin{equation*}
    \frac{1}{\lambda_m} \sum_{i=1}^{m} \int_M \Big(\frac{\phi^2}{\sn^2}-|\nabla\Phi_i|^2 \Big)=\sum_{i=1}^{m} \frac{1}{\lambda_i}\int_M \Big(\frac{\phi^2}{\sn^2}-|\nabla\Phi_i|^2 \Big).
\end{equation*}
Since $M$ is closed, there exists at least one point $p\in M$ such that $|\nabla\Phi_i|^2=0<\frac{\phi^2}{\sn^2}(r)$ at $p$.
Hence, $\lambda_1=\cdots=\lambda_m$.


\end{document}